\documentclass[12pt]{amsart}

\usepackage{amsmath, amssymb}
\usepackage[frame,cmtip,arrow,matrix,line,graph,curve]{xy}
\usepackage{graphpap, color, paralist, pstricks}
\usepackage[mathscr]{eucal}
\usepackage[pdftex]{graphicx}
\usepackage[pdftex,colorlinks,backref=page,citecolor=blue]{hyperref}
\usepackage{tikz}
\usepackage{kotex}
\usepackage{array}
\usepackage{pgfplots}
\usepackage{mathrsfs}
\usetikzlibrary{arrows} 
\usepackage{tikz-cd}
\usepackage{mathtools}

\newcolumntype{P}[1]{>{\centering\arraybackslash}p{#1}}
\newcolumntype{M}[1]{>{\centering\arraybackslash}m{#1}}

\newtheorem{theorem}{Theorem}[section]
\newtheorem{proposition}[theorem]{Proposition}
\newtheorem{corollary}[theorem]{Corollary}
\newtheorem{lemma}[theorem]{Lemma}

\theoremstyle{definition}
\newtheorem{definition}[theorem]{Definition}

\newtheorem{remark}[theorem]{Remark}

\newtheorem{notation}[theorem]{Notation}

\newcommand{\PP}{\mathbb{P}}

\newcommand{\CC}{\mathbb{C}}

\newcommand{\ZZ}{\mathbb{Z}}

\newcommand{\cO}{\mathcal{O} }

\newcommand{\cK}{\mathcal{K} }
\newcommand{\cM}{\mathcal{M} }

\newcommand{\cU}{\mathcal{U} }

\renewcommand{\bar}[1]{\overline{#1}}

\newcommand{\rG}{\mathrm{G} }

\newcommand\bF{\mathbf{F}}

\newcommand\bH{\mathbf{H}}

\newcommand\bS{\mathbf{S}}

\newcommand\hra{\hookrightarrow}

\def\Gr{\mathrm{Gr} }

\def\lr{\rightarrow}

\newcommand{\rank}{\mathrm{rank}\, }

\hypersetup{%
pdftitle={Conics in quintic del Pezzo varieties},%
pdfauthor={Kiryong Chung, Sanghyeon Lee},%
citecolor=blue,%
linkcolor=blue,%
}

\begin{document}

\title{Conics in quintic del Pezzo varieties}
\author{Kiryong Chung}
\address{Department of Mathematics Education, Kyungpook National University, 80 Daehakro, Bukgu, Daegu 41566, Korea}
\email{krchung@knu.ac.kr}

\author{Sanghyeon Lee}
\address{Shanghai Center for Mathematical Sciences, Shanghai, China}
\email{sanghyeon@fudan.edu.cn}


\begin{abstract}
The smooth quintic del Pezzo variety $Y$ is well-known to be obtained as a linear sections of the Grassmannian variety $\mathrm{Gr}(2,5)$ under the Pl\"ucker embedding into $\mathbb{P}^{9}$. Through a local computation, we show the Hilbert scheme of conics in $Y$ for $\text{dim} Y \ge 3$ can be obtained from a certain Grassmannian bundle by a single blowing up/down transformation.
\end{abstract}

\maketitle
\section{Motivation and results}
\subsection{Rational curves in quintic del Pezzo varieties}
By definition, the quintic del Pezzo variety $Y$ is a smooth projective variety such that the anti-canonical line bundle of $Y$ is isomorphic to $-K_{Y}\cong -(\dim Y-1)L$ for some ample generator $L\in \text{Pic}(Y)\cong \ZZ$ with $L^{\dim Y}=5$. By the work of Fujita (\cite{Fuj81}), it is known that the dimension of $Y$ is $\dim Y\leq 6$ and $Y$ is isomorphic to a linear section of Grassmannian variety $\Gr(2,5)$.
In this paper, the authors aim to explain that the Hilbert scheme of conics in the quintic del Pezzo variety can be obtained from a certain Grassmannian bundle. Research on rational curves in the quintic del Pezzo variety has been studied for a long time from a birational geometric perspective (\cite{FN89, Ili94, San14, KPS18, Chu22, Chu23}) . In particular, the existence of rational curves with a specific normal bundles play a crucial role in the determination of the automorphism groups. The main conclusion of this paper is to provide an alternative proof of previous work of authors in \cite[Proposition 5.8 and Proposition 6.7]{CHL18}, which were proved by using Li's results (\cite{Li09}). While the proof in \cite{CHL18} relied on a deformation theoretic argument, our paper will mainly depend on basic linear algebra calculations to provide a proof of the result. The linear algebra calculation is extracted from the second named author's Ph.D. thesis (\cite{Lee18}). Such an approach will offer a good opportunity for comparison with previous proof.
\subsection{Results}
From now on, let us fix $\{e_0,e_1,e_2,e_3, e_4\}$ a standard coordinate vectors of the space $V(\cong \CC^5)$, which provides the projective space $\PP(V)(=\PP^4)$. Let $\mathrm{G}:=\Gr(2,5)$ be the Grassmannian variety of two dimensional subspaces of $V$. Let $\{p_{ij}\}_{0\leq i<j\leq 4}$ be the Pl\"ucker coordinates of $\PP^9=\PP(\wedge^2V)$. Let $\bH_2(X)$ be the Hilbert scheme of conics in a smooth projective variety $X$ with a fixed embedding $X\subset \PP^r$. Let $\cU$ be the universal subbundle over the Grassmannian variety $\mathrm{Gr}(4,V)$. 
Let $\bS(\mathrm{G})=\Gr(3,\wedge^2\cU)$ be the Grassmannian bundle over $\mathrm{Gr}(4,V)$. For the general fiber $(V_3, V_4)\in \bS(\mathrm{G})$, we associate to a conic $\PP(V_3)\cap \Gr(2,V_4)\subset \Gr(2, V)$ as the intersection of $\PP(V_3)$ and $\Gr(2,V_4)$. Then this correspondence provides a birational map \[\Phi: \bS(\mathrm{G})\dashrightarrow \bH_2(\rG).\] The map $\Phi$ is not undefined whenever $\PP(V_3)\subset\Gr(2,V_4)$ and thus the undefined locus $T(\rG)$ of $\Phi$ is isomorphic to the \emph{relative orthogonal} Grassmannian $T(\rG)\cong \text{OG}(3,\wedge^2 \cU)$.
By blowing-up $\bS(\mathrm{G})$ along $T(\rG)$, we have a birational morphism $\widetilde{\Phi}:\widetilde{\bS}(\rG)\lr \bH_2(\rG)$. Furthermore, as applying the Fujiki-Nakano criterion (\cite{FN71}) to the morphism $\widetilde{\Phi}$, it turns out that the extended morphism $\widetilde{\Phi}$ is a smooth blow-down map. In summery,
\begin{theorem}[\protect{\cite[Section 4.1]{CHL18})}]\label{mainthm}
Under the above definition and notations, we have a blow-up/down diagram
\begin{equation}\label{diahil}
\xymatrix{&\widetilde{\bS}(\rG)\ar[dl]_{\Phi}\ar[dr]^{\widetilde{\Phi}}&\\
\bS(\rG)\ar@{=>}[d]_{\pi}&&\bH_2(\rG)\\
\Gr(4,V),&&
}
\end{equation}
where $\pi$ is the canonical bundle morphism.
\end{theorem}
The goal of this paper is to prove that the diagram still holds for quintic del Pezzo $m$-fold $Y_m$ ($m\leq5$). Let $H_1$ (resp. $H_2$) be the linear subspace of $\PP^9$ defined by $p_{12}-p_{03}$ (resp. $p_{13}-p_{24}=0$). Let $Y_{5}=\Gr(2,V)\cap H_1$ and $Y_{4}=\Gr(2,V)\cap H_1\cap H_2$. Then $Y_m$ is smooth since the skew symmetric forms induced from the hyperplane $H_i$ are of rank $4$. Let $\bS(Y_m)$ be the Grassmannian bundle $\Gr(3, \cK_m)$ over $\Gr(4,V)$ where $\cK_m$ is defined by the kernel of composition maps
\begin{equation}\label{ker}
\cK_m:= \mathrm{ker}\big\{\wedge^2\cU \hra \wedge^2 V\otimes\cO_{\Gr(4,V)}\lr \cO_{\Gr(4,V)}^{\oplus^{(6-m)}}\big\},\end{equation}
where the second map is induced from the hyperplanes $H_i$. We can check that the kernel $\cK_m$ is locally free by direct rank computation of the composition map. Let $T(Y_m)=\bS(Y_m)\cap T(\rG)$ be the set-theoretic intersection of $\bS(Y_m)$ and $T(\rG)$. Now we state our main theorem.
\begin{theorem}\label{cleanint1}
Under the above definition and notations, there exists identity
\[
I_{T(Y_m),\; \bS(\rG)} = I_{T(\rG),\; \bS(\rG)} + I_{\bS(Y_m),\; \bS(\rG)}
\]
of ideals in $\bS(\rG)$. That is, $\bS(Y_m)$ and $T(\rG)$ \emph{cleanly} intersect in $\bS(\rG)$.
\end{theorem}
For the detailed discussion of the clean intersection of subvarieties, see the paper \cite{Li09}. Main idea of the proof of Theorem \ref{cleanint1} is to find the defining equation of the intersection part $T(Y)$ in two senses: the set-theoretic and scheme-theoretic intersection by local chart computation, which accompanies lots of linear algebra and brute force. By applying Fujiki-Nakano criterion again (\cite{FN71}), we arrive the same diagram 
\eqref{diahil} as for the case $\rG=\Gr(2,V)$.
\begin{corollary}\label{birmodel}
The digram \eqref{diahil} in Theorem \ref{mainthm} still holds when we replace $\mathrm{G}=\Gr(2,V)$ by the quintic del Pezzo varieties $Y_m$ ($m=4$, $5$).
\end{corollary}
\begin{remark}\label{minorcase}
For the quintic del Pezzo \emph{$3$-fold} $Y_3=\Gr(2,V)\cap H_1\cap H_2\cap H_3$, one can also define a rank $3$-bundle $\cK_{3}$ by using the hyperplanes $H_i$ and thus $\bS(Y_3)=\Gr(4, V)$. In this case $T(Y_3)$ is empty set. That is, the restriction map $\Phi|_{\bS(Y_3)}$ is a morphism. Furthermore one can easily show that $\Phi|_{\bS(Y_3)}$ is an isomorphism (cf. \cite[Proposition 1.2.2]{Ili94} and \cite[Proposition 7.2]{CHL18}). Furthermore, for $m=1,2$, we can easily observe that $\bS(Y_m) = \Gr(3,\cK_m)$ is an empty set because rank of $\cK_m$ is $m$ for general choice of hyperplanes of $\Gr(2,V)$.
\end{remark}
\subsection{Notation and convention}
\begin{itemize}
\item Let us denote by $\Gr(k,V)$ the Grassmannian variety parameterizing $k$-dimensional subspaces in a fixed vector space $V$ with $\dim V=n$.
\item We denote by $\langle e_1, e_2,\cdots, e_k\rangle$ the subspace of $V$ generated by $e_1, e_2,\cdots, e_k\in V$.
\item We sometimes do not distinguish the moduli point $[x]\in \cM$ and the object $x$ parameterized by $[x]$.
\item Whenever the meaning in the context is clear, we will use $I_X$ to denote the ideal $I_{X,Y}$ when $X\subset Y$.
\end{itemize}
\subsection*{Acknowledgements}
The author gratefully acknowledges the many helpful comments of In-Kyun Kim during the preparation of the paper.
\section{Preliminary}
\subsection{Planes in quintic Fano varieties}
If we consider a point $\ell\in \Gr(2,5)$ as a line in $\PP^{4}$ and fix a flag $p\in \PP^1\subset \PP^2\subset \PP^3\subset \PP^{4}$, then a plane in $\Gr(2,5)$ can be represented by one of two types of \emph{Schubert varieties}: $\sigma_{3,1}(p,\PP^3)=\{\ell\,|\, p\in \ell\subset \PP^3\}$ and $\sigma_{2,2}(\PP^2)=\{\ell\,|\, \ell\subset \PP^2\}$. Hence the space $\bF_2(\Gr(2,5))$ (so called, \emph{Fano scheme}) of planes in $\Gr(2,5)$ is isomorphic to 
\[\bF_2(\Gr(2,5))\cong \Gr(1,4,5)\sqcup \Gr(3,5)\]
such that the first (resp. second) one is of $\sigma_{1,3}$ (resp. $\sigma_{2,2}$)-type.
The Fano scheme $\bF_2(Y)$ of planes in a quintic Fano variety $Y$ was studied by several authors (\cite{Ili94, Pro94}). Let
\[\bF_2(Y)=\bF_2^{3,1}(Y)\sqcup \bF_2^{2,2}(Y)\]
be the disjoin union of two connected components such that $F_2^{3,1}(Y)$ parametrizes $\sigma_{3,1}$-type planes in $Y$ and $F_2^{2,2}(Y)$ parametrizes $\sigma_{2,2}$-type planes in $Y$.
\begin{proposition}(\cite[Section 4.4]{Ili94})\label{sy552}
Let $Y_5=\Gr(2,5)\cap H_1$. The first component $\bF_2^{3,1}(Y_5)$ is isomorphic to the blown-up of the projective space $\PP^4$ at a point and $\bF_2^{2,2}(Y_5)$ is isomorphic to a smooth quadric threefold $\Sigma$.
\end{proposition}

\begin{proposition}[\protect{\cite[Proposition 2.2]{Pro94}}]\label{planespace1}
Let $Y_4=\Gr(2,5)\cap H_1\cap H_2$. The space $\bF_2^{3,1}(Y_4)$ is isomorphic to a smooth conic $C_{v}:=\{[a_0:a_1:a_2:a_3:a_4]\mid a_0a_4+a_1^2=a_2=a_3=0\}\subset \PP(V)$ and $\bF_2^{2,2}(Y_4)$ is isomorphic to a point $[S]$.
\end{proposition}
\begin{remark}\label{planeeq}
By the proof of \cite[Lemma 6.3]{CHL18}, the $\sigma_{3,1}$-type planes $P_t$ in $Y_4$ parameterized by $t\in C_{v}$ are $P_t=\PP(V_1\wedge V_4)$ where $V_1=\langle e_0+te_1-t^2e_4\rangle$ and $V_4=\langle e_0,e_1,e_2+te_3,e_4\rangle$. Also the unique plane $S$ in $Y_4$ is given by $S=\PP(\wedge^2 V_3)$ such that $V_3=\langle e_0,e_1,e_4\rangle$.
\end{remark}

\subsection{Conics via a Grassmannian bundle}
Let $\cU$ be the universal subbundle over the Grassmannian $\mathrm{Gr}(4,V)$. 
Let $\bS(\mathrm{G})=\Gr(3,\wedge^2\cU)$ be the Grassmannian bundle over $\mathrm{Gr}(4,V)$. The space $\bS(\rG)$ is an incidence variety of pairs 
\[\bS(\rG)=\{(U,V_4)\,|\,U\subset \wedge^2 V_4\}\subset \mathrm{Gr}(3,\wedge^2 V)\times \mathrm{Gr}(4, V).\]
Furthermore the correspondence
\[
(U,V_4)\mapsto \PP(U)\cap \Gr(2,V_4)
\]
between $\bS(\rG)$ and $\bH_2(\rG)$ provides an birational map $\Phi: \bS(\rG)\dashrightarrow \bH_2(\rG)$.
\begin{lemma}\label{interG}
Let $T(\rG)$ be the undefined locus of the map $\Phi$. Then $T(\rG)$ is isomorphic to the disjoint union of flag varieties:
\begin{equation}\label{undefG}
T(\rG)\cong\Gr(1,4,5)\sqcup \Gr(3,4,5).
\end{equation}
\end{lemma}
\begin{proof}
Fiberwisely, it is clearly that $\Phi$ is not defined if and only if $\PP(U)\subset \Gr(2,V_4)$. Hence $T(\rG)$ is isomorphic to the relative orthogonal Grassmannian $T(\rG)\cong \textrm{OG}(3,\wedge^2\cU)$, where the later space is the disjoint union $\Gr(1,4,5)\sqcup \Gr(3,4,5)$ of the two flag varieties.
\end{proof}
\begin{notation}
$T^{3,1}(\rG) := \Gr(1,4,5)$ and $T^{2,2}(\rG) := \Gr(3,4,5)$ in equation \eqref{undefG}.
\end{notation}
The embedding $T(\rG)=T^{3,1}(\rG)\sqcup T^{2,2}(\rG)\hookrightarrow \bS(\rG)$ is defined by the following ways.
\begin{enumerate}[(a)]
\item For a pair $(V_1,V_4) \in T^{3,1}(\rG)$ ($V_1$ is a $1$-dimensional vector space representing a \emph{vertex point} of $\sigma_{3,1}$-plane), 
$$(V_1,V_4)\mapsto (W, V_4)$$ 
where $W=\ker(\wedge^2V_4\twoheadrightarrow \wedge^2(V_4/V_1))(=V_1\wedge V_4)$ is the $3$-dimensional vector space. In this case, $(V_1,V_4)$ determines a $\sigma_{3,1}$-type plane.
\item For a pair $(V_3,V_4)\in T^{2,2}(\rG)$, 
$$(V_3,V_4)\mapsto (\wedge^2V_3, V_4).$$
In this case, $V_3$ determines a $\sigma_{2,2}$-type plane.
\end{enumerate}

\subsection{Determinant of matrix product}
We recall the Cauchy-Binet formula here, which is useful for further calculation.
\begin{proposition}[\protect{Cauchy-Binet formula}]
Let $A$ (resp. $B$) be a $n\times m$ (resp. $m \times n$) matrix where $n \leq m$. Then we have the following formula for the determinant of the matrix $AB$:
\[
\det(AB)=\sum\limits_{S \in \binom{[m]}{n}}\det A_{[n],S}\cdot \det B_{S,[n]}
\]
where $[m]={1,2,...,m}$ is a set and $\binom{[m]}{n}$ is a set of $n$ combinations of elements in $[m]$.
\end{proposition}
For $n=2, m=3$ case, we can check the following corollary by direct calculation.
\begin{corollary}[\protect{\cite[Example 4.9]{BW89}}]\label{Cauchybinetcoro}
Let $A$ (resp. $B$) be a $2\times 3$ (resp. $3\times 2$) matrix. Let $[A]_0$, $[A]_1$ (resp. $[B]^0$, $[B]^1$) be a row (resp. column) vector of $A$ (resp. $B$). Then
\[ \det {AB} = ([A]_0\times[A]_1) \cdot ([B]^0\times[B]^1),
\]
where $'\times'$ is a \emph{cross product} defined in $\CC^3$.
\end{corollary}
\section{Proof of Theorem \ref{cleanint1}}
By Remark \ref{minorcase}, for $m=1,2,3$, we have nothing to prove for Theorem \ref{cleanint1}. So we will consider $m=4,5$ case in this section. Recall the definition of the bundle $\cK_{m}$ in equation \eqref{ker}, which is inherited from the hyperplanes $H_1$ and $H_2$ of $\PP(\wedge^2 V_5)$.

\begin{definition}
Let
\begin{itemize}
\item $\bS(Y):=\mathrm{Gr}(3, \cK_{m})\subset \mathrm{Gr}(3,\wedge^2\cU)=\bS(\rG)$ and
\item $T(Y):= \bS(Y)\cap T(\rG)$ in $\bS(\rG)$
\end{itemize}
for $Y=Y_m$, $m=4$ or $5$.
\end{definition}
Note that the scheme structure of $\bS(Y)$ and $T(\rG)$ are reduced ones. In this section, we prove our main Theorem \ref{cleanint1}. Firstly, in Section \ref{setinter} and Section \ref{schthe2}, we will describe the intersection part $T(Y)$ set-theoretically. Secondly, in Section \ref{schthe1} (resp. Section \ref{schint4}), we confirm the clean intersection of $\bS(Y)$ and $T(\rG)$ for $Y=Y_5$ (resp. $Y_4$). Note that the explicit computations in the following sections are extracted from the second named author's Ph.D. thesis (\cite[Section 4.3.3, 4.3.4, 4.4.3]{Lee18}).

\subsection{Set theoretic intersection of $S(Y_5)$ and $T(\rG)$}\label{setinter}
Let $\Omega := p_{12}-p_{03}$ be the rank $4$ skew-symmetric $2$ form on $V(=\CC^5)$ induced from the hyperplane $H_1$.
\begin{proposition}\label{planes}
The intersection part $T(Y_5)$ is a fiberation over $\mathrm{Gr}(3,4)$ linearly embedded in $\mathrm{Gr}(4,5)$, where the linear embedding is given by the $1$-$1$ correspondence between $3$-dimensional subspaces in $\CC^5/\langle e_4 \rangle$ and $4$-dimensional subspaces in $V$ containing $\langle e_4 \rangle$. Furthermore,
\begin{enumerate}
\item the restriction $\Omega|_{V_4}$ on $V_4$ becomes a rank $2$ singular two form for each $V_4 \in \mathrm{Gr}(3,4) \subset \mathrm{Gr}(4,5)$.
\item The fiber of $T^{3,1}(Y^5)\subset \Gr(1,4,5)$ (resp. $T^{2,2}(Y^5)\subset \Gr(3,4,5)$) over $V_4$ is canonically identified with $\PP(\ker \Omega|_{V_4})\cong \PP^1 \subset \PP(V)$ (resp. $\PP((V/\ker \Omega|_{V_4})^*) \cong \PP^1 \subset \PP((V_4)^*)$).
\end{enumerate}
\end{proposition}

\begin{proof}
Case (1): Consider an arbitrary 4-dimensional vector space $V_4 \in \mathrm{Gr}(4,5)$. We can observe that rank $\Omega|_{V_4}\geq 2$ since we have $\rank \Omega =4$ and $\rank\Omega \leq \rank \Omega|_{V_4} + 2$.
If $\rank \Omega|_{V_4}=4$, then there cannot exist a vector $v\in \CC^5$ such that $v$ is orthogonal to $V_4$ with respect to the 2-form $\Omega$. Hence there does not exist any $\sigma_{3,1}$-plane contained in the fiber of $T(Y^5)$ on $V_4$. Moreover, there cannot exist a $3$-dimensional subspace $V_3 \subset V_4$ of $V_4$ such that $\Omega|_{V_3}=0$ since we have $\rank \Omega|_{V_4} \leq \rank \Omega|_{V_3}+2$. Therefore, there is no $\sigma_{3,1}$-plane in the fiber of $T(Y^5)$ over $V_4$. In summary, the fiber of $T(Y^5)$ over $V_4$ is empty whenever $\rank\Omega|_{V_4}=4$. Next, consider the case when $\rank \Omega|_{V_4}=2$. Assume that $V\cap \ker \Omega = V\cap \langle e_4 \rangle=\langle 0 \rangle$. Then, since $\Omega = p_{12}-p_{03}$ descent to the rank $4$ skew-symmetric 2-form $\bar{\Omega}$ on quotient space $V/\langle e_4 \rangle$. Since $V_4 \cap \langle 0 \rangle = 0$, we can easily observe that the natural isomorphism $\phi : V_4 \stackrel{\cong}{\to} V/\langle e_4 \rangle$ preserves skew-symmetric two forms, i.e. $\phi^* \bar{\Omega} = \Omega|_{V_4}$. Therefore $\rank \Omega|_{V_4}=4$, which is a contradiction. Thus we have $\langle e_4 \rangle \subset \CC^5$. Conversely if $\langle e_4 \rangle \subset \CC^5 = \ker \Omega$, then we have $\rank \Omega = 2$. Therefore $\rank \Omega|_{V_4}=2$ if and only if $V_4 \in \mathrm{Gr}(3,4)\subset \mathrm{Gr}(4,5)$, where $\mathrm{Gr}(3,4)\subset \mathrm{Gr}(4,5)$ is a linear embedding given by the 1-1 correspondence between $3$-dimensional subspaces in $\CC^5/\langle e_4 \rangle$ and $4$-dimensional subspaces in $\CC^5$ containing $e_4$. 

Case (2): The fiber of $T^{3,1}(Y^5) \subset F(1,4,5)$ over $V_4 \in \mathrm{Gr}(3,4)\subset \mathrm{Gr}(4,5)$ is represented by pairs $(p,V_4)$ such that $\Omega(p,V_4)=0$. Therefore, the fiber is canonically identified with $\PP(\ker \Omega)\cong \PP^1 \subset \PP(\CC^5)$. The fiber of $T^{2,2}(Y^5) \subset F(3,4,5)$ over $V_4$ is represented by pairs $(V_3,V_4)$ such that $V_3 \subset V_4$, $\Omega|_{V_3}=0$. Assume that $V_3 \cap \ker \Omega|_{V_4} = 1$. Then there is a natural isomorphism $\phi : V_3/(V_3 \cap \ker \Omega|_{V_4}) \stackrel{\cong}{\to} V_4/\ker \Omega|_{V_4}$. Then, when we denote by $\bar{\Omega}$ the induced 2-form on $V_4/\ker \Omega|_{V_4}$, and $\bar{\Omega'}$ be the induced 2-form on $V_3/(V_3 \cap \ker \Omega|_{V_4})$, we can observe that $\phi^* \bar{\Omega} = \bar{\Omega'}$. But we have $\rank \bar{\Omega'}=0$ since $\rank{\Omega|_{V_3}=0}$ and $\rank\bar{\Omega}=2$ since $\rank \Omega|_{V_4}=2$, which leads to the contradiction. Therefore, we have $\ker \Omega|_{V_4} \subset V_3$. Conversely, if $\ker \Omega|_{V_4} \subset V_3$, then it is clear that $\rank \Omega|_{V_3}=0$. Therefore, the fiber is canonically identifed with $\PP((V_4/\ker \Omega|_{V_4})^*) \cong \PP^1 \subset \PP((\CC^5)^*)$.
\end{proof}
\subsection{Clean intersection of $S(Y_5)$ and $T(\rG)$}\label{schthe1}
This subsection is devoted to proving Theorem \ref{cleanint1} for the $5$-fold $Y_5$. By Lemma \ref{interG}, we know that $T(\rG)$ is an $\text{OG}(3,6)\cong \PP^3\sqcup \PP^3$-bundle over $\mathrm{Gr}(4,5)$, $\sigma_{3,1}$-planes and $\sigma_{2,2}$-planes corresponds to each disjoint $\PP^3$. Denote them by $T(\rG)_{2,2}$ and $T(\rG)_{3,1}$. Since they are disjoint, we can consider them independently, i.e. it is enough to show that $I_{T(Y)_{2,2}, \bS(\rG)}= I_{\bS(Y), \bS(\rG)}+I_{T(\rG)_{2,2}, S(\rG)}$, $I_{T(Y)_{3,1}, \bS(\rG)}= I_{\bS(Y),\bS(\rG)}+I_{T(\rG)_{3,1},\bS(\rG)}$ where $T(Y)_{2,2}:=T(\rG)_{2,2}\cap \bS(Y)$, $T(Y)_{3,1}:=T(\rG)_{3,1}\cap \bS(Y)$.

We check $I_{T(Y)_{2,2}, \bS(\rG)}= I_{S(Y),\bS(\rG)}+I_{T(\rG)_{2,2},\bS(\rG)}$ for affine local charts.
Consider a chart for $\bS(\rG)$. Since $\bS(\rG)$ is a $\mathrm{Gr}(3,6)$-bundle over $\mathrm{Gr}(4,5)$, we should consider chart for $\Lambda \in \mathrm{Gr}(4,5)$ and $F \in \mathrm{Gr}(3,6) = \mathrm{Gr}(3,\wedge^2 \Lambda)$.
There are $5$ standard charts for $\Lambda \in \mathrm{Gr}(4,5)$:
\[
\Lambda= 
\left(\begin{matrix}
1 & 0 & 0 & 0 & a \\
0 & 1 & 0 & 0 & b \\
0 & 0 & 1 & 0 & c \\
0 & 0 & 0 & 1 & d \\ 
\end{matrix}\right),
\Lambda= 
\left(\begin{matrix}
1 & 0 & 0 & a & 0 \\
0 & 1 & 0 & b & 0 \\
0 & 0 & 1 & c & 0 \\
0 & 0 & 0 & d & 1 \\ 
\end{matrix}\right),
\Lambda= 
\left(\begin{matrix}
1 & 0 & a & 0 & 0 \\
0 & 1 & b & 0 & 0 \\
0 & 0 & c & 1 & 0 \\
0 & 0 & d & 0 & 1 \\ 
\end{matrix}\right),
\] 
\[
\Lambda= 
\left(\begin{matrix}
1 & a & 0 & 0 & 0 \\
0 & b & 1 & 0 & 0 \\
0 & c & 0 & 1 & 0 \\
0 & d & 0 & 0 & 1 \\ 
\end{matrix}\right)
\textrm{ and }
\Lambda= 
\left(\begin{matrix}
a & 1 & 0 & 0 & 0 \\
b & 0 & 1 & 0 & 0 \\
c & 0 & 0 & 1 & 0 \\
d & 0 & 0 & 0 & 1 \\ 
\end{matrix}\right).
\]
But in the first chart:
\[ \Lambda = \left(\begin{matrix}
1 & 0 & 0 & 0 & a \\
0 & 1 & 0 & 0 & b \\
0 & 0 & 1 & 0 & c \\
0 & 0 & 0 & 1 & d \\ \end{matrix} \right)
\]
the equation of $Y^5: p_{12}-p_{03}$ has no solution. Furthermore, since the symmetry interchanging the index $1$, $2$ and $0$, $3$ does not change the equation $p_{12}-p_{03}$, it is enough to consider the following two chart of $\mathrm{Gr}(4,5)$:
\[
\Lambda =
\left( \begin{matrix}
1 & 0 & a & 0 & 0  \\
0 & 1 & b & 0 & 0  \\
0 & 0 & c & 1 & 0  \\
0 & 0 & d & 0 & 1  \\ 
\end{matrix} \right)
\textrm{ and }
\Lambda =
\left( \begin{matrix}
1 & 0 & 0 & a & 0 \\
0 & 1 & 0 & b & 0 \\
0 & 0 & 1 & c & 0 \\
0 & 0 & 0 & d & 1 \\ 
\end{matrix} \right).
\]
Let us start with the first chart:
\[
\Lambda = \left( \begin{matrix}
1 & 0 & a & 0 & 0  \\
0 & 1 & b & 0 & 0  \\
0 & 0 & c & 1 & 0  \\
0 & 0 & d & 0 & 1  \\ \end{matrix} \right).
\]
Let $q_{01},...,q_{23}$ be a coordinate of a fiber of $\wedge^2\cU$ over this chart, where $\cU$ is a tautological rank $4$ bundle over $\mathrm{Gr}(4,5)$. Then we have $p_{12}-p_{03} = -aq_{01}-q_{02}+cq_{12}+dq_{13}$.
By Proposition \ref{planes}, $T^{2,2}(Y)$ is a fibration over $\mathrm{Gr}(3,4)$ linearly embedded in $\mathrm{Gr}(4,5)$, whose images are $\Lambda \in \mathrm{Gr}(4,5)$ such that $e_4 \in \Lambda$. Therefore, we have equation $d=0$ in $I_{T(Y)_{2,2}}$.

Next, $\sigma_{2,2}$-plane corresponds to $\PP^2$-plane in $\PP\Lambda\cong \PP^3 \subset \PP^4$ must one be of the following form(i.e. it correspond to the row space of the matrix $R\cdot \Lambda$) : 
\[
R = 
\left( \begin{matrix}
 1 & 0 & 0 & \alpha \\ 
 0 & 1 & 0 & \beta \\
 0 & 0 & 1 & \gamma \\
\end{matrix}
\right)
or
\left( \begin{matrix}
 1 & 0 & \alpha  & 0 \\ 
 0 & 1 & \beta   & 0 \\
 0 & 0 & \gamma  & 1 \\
\end{matrix}
\right)
or
\left( \begin{matrix}
 1 & \alpha  & 0 & 0 \\ 
 0 & \beta   & 1 & 0 \\
 0 & \gamma  & 0 & 1 \\
\end{matrix}
\right)
or
\left( \begin{matrix}
 \alpha  & 1 & 0 & 0 \\ 
 \beta   & 0 & 1 & 0 \\
 \gamma  & 0 & 0 & 1 \\
\end{matrix}
\right).
\]

Therefore, the intersection of $S(Y)$ and $T(\rG)$ arises only in the following three charts for fibers $F \in \mathrm{Gr}(3, \wedge^2 \Lambda)$ : 
\[
F=
\bordermatrix{
&01& 02& 03& 12& 13& 23\cr
&1 & 0 & e & 0 & f & g \cr
&0 & 1 & h & 0 & i & j \cr
&0 & 0 & k & 1 & l & m 
}, 
F=
\bordermatrix{
&01& 02& 03& 12& 13& 23\cr
&1 & e & 0 & f & 0 & g \cr
&0 & h & 1 & i & 0 & j \cr
&0 & k & 0 & l & 1 & m 
},
\]
\[
F=
\bordermatrix{
&01& 02& 03& 12& 13& 23\cr
&e & 1 & 0 & f & g & 0 \cr
&h & 0 & 1 & i & j & 0 \cr
&k & 0 & 0 & l & m & 1 
}
\textrm{and }F=
\bordermatrix{
&01& 02& 03& 12& 13& 23\cr
&e & f & g & 1 & 0 & 0 \cr
&h & i & j & 0 & 1 & 0 \cr
&k & l & m & 0 & 0 & 1 
}
\]
where the upper indices are indices of Pl\"ucker coordinates.
Let us start with the first chart :
\[
F=
\left(\begin{matrix}
1 & 0 & e & 0 & f & g \\
0 & 1 & f & 0 & i & j \\
0 & 0 & g & 1 & l & m \\ 
\end{matrix}\right), 
\]

In this case, we can easily observe that a $\sigma_{2,2}$-plane contained in this chart must correspond to the row space of a matrix of the form :
\[
R\Lambda = 
\left( \begin{matrix}
 1 & 0 & 0 & \alpha \\ 
 0 & 1 & 0 & \beta \\
 0 & 0 & 1 & \gamma \\
\end{matrix}
\right)\cdot \Lambda.
\]

For a matrix $M$, we let $M_i^j$ be a matrix obtained from $M$ by deleting $i$-th row and $j$-th column. From the equation $p_{12}-p_{03}=0$, and since $d=0$ for $\sigma_{2,2}$-planes in $T(Y)_{2,2}$, using Corollary \ref{Cauchybinetcoro}, we can observe that the equation for $T(Y)_{2,2}$ in this chart is $([R^4]_i\times [R^4]_j) \cdot ([\Lambda_4^5]^1\times [\Lambda_4^5]^2 - [\Lambda_4^5]^0 \times [\Lambda_4^5]^3) = ([R^4]_i\times [R^4]_j) \cdot (c,1,-a) = 0$ for all $0 \leq i < j \leq 2$.
Then, since $([R^4]_0\times [R^4]_2)=(0,-1,0)$, we have no solution. Therefore, the intersection of $T(\rG)$ and $S(Y)$ does not happens in this chart.

Next, we consider the second chart:
\[
F=
\left(\begin{matrix}
1 & e & 0 & f & 0 & g \\
0 & f & 1 & i & 0 & j \\
0 & g & 0 & l & 1 & m \\
\end{matrix}\right).
\]
We can easily observe that a $\sigma_{2,2}$-plane contained in this chart must correspond to the row space of a matrix of the form:
\[
R\Lambda = 
\left( \begin{matrix}
 1 & 0 & \alpha  & 0 \\ 
 0 & 1 & \beta   & 0 \\
 0 & 0 & \gamma  & 1 \\
\end{matrix}
\right)
\cdot \Lambda.
\]

In the same manner we can show that $([R^4]_i\times [R^4]_j) \cdot ([\Lambda_4^5]^1\times [\Lambda_4^5]^2 - [\Lambda_4^5]^0 \times [\Lambda_4^5]^3) = ([R^4]_i\times [R^4]_j) \cdot (c,1,-a) = 0$ for all $0 \leq i < j \leq 2$ is the equation for $T(Y)_{2,2}$ in this chart under the condition $d=0$. By direct calculations, we have $\gamma=0, \alpha c + \beta + a = 0$.

We observe that this $\sigma_{2,2}$-plane which correspond to the row space of the matrix $R\Lambda$ correspond to the following matrix form in the chart of $F$:
\[
\left(\begin{matrix}
1 & \beta  & 0 & -\alpha & 0 & 0      \\
0 & \gamma & 1 & 0       & 0 & \alpha \\
0 & 0      & 0 & \gamma  & 1 & \beta  \\
\end{matrix}\right)
\]
In summary, we obtain the full description of the equation of $T^{2,2}(Y)$ in the chart:
\[
I_{T^{2,2}(Y)} = \langle g,i,k, f+j, e-m, h,l,d, -fc+e+a \rangle .
\]

On the other hand, from the equation $-aq_{01}-q_{02}+cq_{12}+dq_{13}$, we have
\[
I_{S(Y)} = \langle -a-e+cf,-h+ci,-k+cl+d \rangle.
\]

And clearly the equation for $T(\rG)_{2,2}$ is given by
\[
I_{T^{2,2}(\rG)}=\langle g,i,k,f+j,e-m,h-l \rangle.
\]

Therefore, we can check the following clean intersection by direct calculation
\[
I_{T(\rG)_{2,2}} + I_{S(Y)} = I_{T(Y)_{2,2}}.
\]

Next, we consider the third chart :
\[
F=
\left(\begin{matrix}
e & 1 & 0 & f & g & 0 \\
h & 0 & 1 & i & j & 0 \\
k & 0 & 0 & l & m & 1 \\
\end{matrix}\right).
\]

Then we can easily observe that a $\sigma_{2,2}$-plane contained in this chart must correspond to the row space of a matrix of the form
\[
R\Lambda = 
\left( \begin{matrix}
 1 & \alpha  & 0 & 0 \\ 
 0 & \beta   & 1 & 0 \\
 0 & \gamma  & 0 & 1 \\
\end{matrix}
\right)
\cdot \Lambda.
\]
Then in the same manner, we can calculate $I_{T(Y)_{2,2}}$, $I_{S(Y)}$ and $I_{T(\rG)_{2,2}}$ by direct calculation :
\begin{align*}
& I_{T(Y)_{2,2}} = \langle g,i,k,f-j,e-m,h,l,d,fc-1-ea \rangle \\
& I_{S(Y)} = \langle -ae + cf + dg -1, -ah + ci + dj, -ak + dl + dm \rangle \\
& I_{T(\rG)_{2,2}}= \langle g,i,k, f-j, e-m, h+l \rangle.
\end{align*}
Therefore we can check the clean intersection $I_{T^{2,2}(Y)} = I_{S(Y)} + I_{T^{2,2}(\rG)}$ by direct calculation.

At last, we consider the fourth chart:
\[
F=
\left(\begin{matrix}
e & f & g & 1 & 0 & 0 \\
h & i & j & 0 & 1 & 0 \\
k & l & m & 0 & 0 & 1 \\
\end{matrix}\right).
\]
Then we can easily observe that a $\sigma_{2,2}$-plane contained in this chart must correspond to the row space of a matrix of the form:
\[
R\Lambda = 
\left( \begin{matrix}
 \alpha  & 1 & 0 & 0 \\ 
 \beta   & 0 & 1 & 0 \\
 \gamma  & 0 & 0 & 1 \\
\end{matrix}
\right)
\cdot \Lambda.
\]
Then in the same manner, we can calculate $I_{T^{2,2}(Y)}$, $I_{S(Y)}$ and $I_{T^{2,2}(\rG)}$ by direct calculation:
\begin{align*}
& I_{T^{2,2}(Y)} = \langle g,i,k,f-j,e+m,h,l,d,c-f-ae \rangle \\
& I_{\bS(Y)} = \langle -ae+c-f, -ah+d-i, -ak-l \rangle \\
& I_{T^{2,2}(\rG)}= \langle g,i,k, f-j, e+m, h-l \rangle 
\end{align*}
Therefore we can check the clean intersection $I_{T^{2,2}(Y)} = I_{S(Y)} + I_{T^{2,2}(\rG)}$ by direct calculation.

In summary, we checked the clean intersection $I_{S(Y)} + I_{T^{2,2}(G)}$ for the chart 
\[\Lambda = 
\left( \begin{matrix}
1 & 0 & a & 0 & 0  \\
0 & 1 & b & 0 & 0  \\
0 & 0 & c & 1 & 0  \\
0 & 0 & d & 0 & 1  \\ 
\end{matrix} \right)\in \mathrm{Gr}(4,5).\]
and all charts for $F\in \mathrm{Gr}(3,\wedge^2\Lambda)$.

We can also check the clean intersection for the second chart: 
\[
\Lambda = \left( \begin{matrix}
1 & 0 & 0 & a & 0 \\
0 & 1 & 0 & b & 0 \\
0 & 0 & 1 & c & 0 \\
0 & 0 & 0 & d & 1 \\ 
\end{matrix} \right).
\] 
But the computation proceeds exactly in the same manner as the case of first chart so we do not write it down here.

Next, we can also check clean intersection at $T(Y)_{3,1}$. We should check $I_{T^{3,1}(Y), \bS(\rG)}= I_{\bS(Y), \bS(\rG)}+I_{T^{3,1}(G), \bS(\rG)}$. We first consider an open chart for $\bS(\rG)$. Same as in the case of $T(Y)_{2,2}$, it is enough to consider $2$ chart for $\Lambda$:
\[
\Lambda = \left( \begin{matrix}
1 & 0 & a & 0 & 0  \\
0 & 1 & b & 0 & 0  \\
0 & 0 & c & 1 & 0  \\
0 & 0 & d & 0 & 1  \\ \end{matrix} \right)
\textrm{ and }
\Lambda = \left( \begin{matrix}
1 & 0 & 0 & a & 0 \\
0 & 1 & 0 & b & 0 \\
0 & 0 & 1 & c & 0 \\
0 & 0 & 0 & d & 1 \\ \end{matrix} \right)  
.\]
Let us start with the first chart :
\[
\Lambda = \left( \begin{matrix}
1 & 0 & a & 0 & 0  \\
0 & 1 & b & 0 & 0  \\
0 & 0 & c & 1 & 0  \\
0 & 0 & d & 0 & 1  \\ \end{matrix} \right).
\]
Let $q_{01},...,q_{23}$ be a coordinate of a fiber of $\wedge^2\cU$ over this chart. Then we have $p_{12}-p_{03} = -aq_{01}-q_{02}+cq_{12}+dq_{13}$.

Next, by Proposition \ref{planes}, $T^{3,1}(Y)$ is a fibration over $\mathrm{Gr}(3,4)$ linearly embedded in $\mathrm{Gr}(4,5)$, whose images are $\Lambda \in \mathrm{Gr}(4,5)$ such that $e_4 \in \Lambda$. Therefore, we have equation $d=0$ in $I_{T^{3,1}(Y)}$. Furthermore, by Proposition \ref{planes} again, a pair $(x,\Lambda)\in T^{3,1}(\rG)$ over $\Lambda$ contained in $T^{3,1}(Y)$ if and only if the vertex $x$ must be contained in the projectivized kernel of the $2$-form $(-ap_{01}+cp_{12}+dp_{13}-p_{02})$, which is equal to $\PP^1=\PP\langle (c,1,-a,0),(0,0,0,1) \rangle$. Therefore we should consider two types of the vertex $x$: $x = (c,1,-a,s) \textrm{ and } x = (sc,s,-sa,1)$ where $s \in k$. Let us start with the first vertex type: $x = (c,1,-a,s)$. The corresponding $\sigma_{3,1}$-plane is spanned by $(c,1,-a,s)\wedge(1,0,0,0), (c,1,-a,s)\wedge(0,0,1,0), (c,1,-a,s)\wedge(0,0,0,1)$. So we can rewrite it by a following $3\times 6$-matrix:
\[
\left(\begin{matrix}
1&-a&s&0&0&0\\ 
0& c&0&1&0&-s\\ 
0& 0&c&0&1&-a
\end{matrix}\right).
\]
Thus, intersection of $S(Y)$ and $T^{3,1}(G)$ only occurs in the following chart of $F$:
\[
F=\left(\begin{matrix}
1&e&f&0&0&g\\ 
0&h&i&1&0&j\\ 
0&k&l&0&1&m
\end{matrix}\right).
\] 
Therefore, we have $I_{T(Y)_{3,1}}=\langle f+j, e-m, e+a, h-l, c-h, g,i,k,d \rangle$.

On the other hand, $\sigma_{3,1}$-plane contained in this chart of $F$ is defined by the vertex of the form $x = (\alpha,1,\beta,\gamma)$ which correspond to the following $3 \times 6$-matrix:
\[
\left(\begin{matrix}
1 & \beta  & \gamma & 0 & 0 & 0       \\ 
0 & \alpha & 0      & 1 & 0 & -\gamma \\ 
0 &       0& \alpha & 0 & 1 & \beta   \\
\end{matrix}\right).
\]
Thus, we have $I_{T(\rG)_{3,1}} = \langle f+j, e-m, h-l, g,i,k \rangle$. Furthermore, from the equation $-aq_{01}-q_{02}+cq_{12}+dq_{13}$, we obtain the equation for $S(Y)$, i.e. $I_{S(Y)} = \langle -a-e, c-h, d-k \rangle$.
Finally, we can check the clean intersection $I_{T(Y)_{3,1}} = I_{T(\rG)_{3,1}} + I_{S(Y)}$ by direct calculation.

Next, we consider the second vertex type $x = (sc,s,-sa,1)$. The corresponding $\sigma_{3,1}$-plane is spanned by $(sc,s,-sa,1)\wedge(1,0,0,0), (sc,s,-sa,1)\wedge(0,1,0,0),(sc,s,-sa,1)\wedge(0,0,1,0)$. So we can rewrite it by a following $3\times 6$-matrix:
\[
\left(\begin{matrix}
s   & -sa & 1 & 0   & 0 & 0 \\ 
-sc & 0   & 0 & -sa & 1 & 0 \\ 
0   & -sc & 0 & -s  & 0 & 1 \\
\end{matrix}\right).
\]
Thus the intersection of $S(Y)$ and $T^{3,1}(\rG)$ only occurs in the following chart of $F$:
\[
F=\left(\begin{matrix}
e & f & 1 & g & 0 & 0 \\ 
h & i & 0 & j & 1 & 0 \\ 
k & l & 0 & m & 0 & 1 \\
\end{matrix}\right).
\] 
Therefore, we have $I_{T(Y)_{3,1}}=\langle f-j,h-l,e+m,g,i,k,f+ea,l-cm,d \rangle$.

On the other hand, $\sigma_{3,1}$-plane contained in this chart of $F$ is defined by the vertex of the form $x = (\alpha,\beta,\gamma,1)$ which correspond to the following $3 \times 6$-matrix:
\[
\left(\begin{matrix}
\beta   & \gamma  & 1 & 0       & 0 & 0 \\ 
-\alpha & 0       & 0 & \gamma  & 1 & 0 \\ 
0       & -\alpha & 0 & -\beta  & 0 & 1 \\
\end{matrix}\right).
\]
Thus, we have $I_{T(\rG)_{3,1}} = \langle f-j,h-l,e+m,g,i,k \rangle$. Furthermore, from the equation $-aq_{01}-q_{02}+cq_{12}+dq_{13}$, we obtain the equation for $S(Y)$, i.e. $I_{S(Y)} = \langle -ae-f+eg,-ah-i+cj+d,-ak-l+cm \rangle$. Finally, we can check the clean intersection $I_{T(Y)_{3,1}} = I_{T(\rG)_{3,1}} + I_{S(Y)}$ by direct calculation.

In summary, we checked the clean intersection $I_{S(Y)} + I_{T(\rG)_{3,1}}=I_{T(Y)_{3,1}}$ for the chart \[\Lambda = 
\left( \begin{matrix}
1 & 0 & a & 0 & 0  \\
0 & 1 & b & 0 & 0  \\
0 & 0 & c & 1 & 0  \\
0 & 0 & d & 0 & 1  \\ 
\end{matrix} \right).\]
We can also check the clean intersection for the second chart. But the it proceeds exactly in the same manner as the case of first chart so we do not write it down here. In summary, we checked the clean intersection of $\bS(Y)$ and $T(\rG)$ in $\bS(\rG)$ by direct calculation.
\subsection{Set theoretic intersection of $S(Y_4)$ and $T(\rG)$}\label{schthe2}
Let $\Omega_1 := p_{12}-p_{03}$ and $\Omega_2 := p_{13}-p_{24}$ be the skew-symmetric 2-forms on $V(=\CC^5)$ induced from the hyperplane $H_1$ and $H_2$ respectively.
\begin{proposition}\label{planes2}
The intersection part $T(Y^4)$ is a double cover over $\PP^1\cong \mathrm{Gr}(1,2)\subset \mathrm{Gr}(4,5)$, with 2 connected components, where $\mathrm{Gr}(1,2)\subset \mathrm{Gr}(4,5)$ is a linear embedding given by 1-1 correspondence between $1$-dimensional subspaces in $\CC^5/\langle e_0, e_1, e_4 \rangle $ and $4$-dimensional subspaces in $\CC^5$ containing $\langle e_0,e_1,e_4 \rangle$. Furthermore,
\begin{enumerate}
\item the restriction $\Omega_{i}|_{V_4}$ on $V_4$ for $i=1$, $2$ become a rank $2$ singular two form for each $V_4 \in \mathrm{Gr}(1,2) \subset \mathrm{Gr}(4,5)$.
\item The fiber of $T(Y^4)$ over $V_4 \in \mathrm{Gr}(1,2)$ is a $2$-point set, one point is the fiber of $T^{3,1}(Y^4) \subset \Gr(1,4,5)$ over $V_4$ defined by a pair $(\ker \Omega_1|_{V_4} \cap \ker \Omega_2|_{V_4}, V_4)$, and the other point is a fiber of $T^{2,2}(Y^4) \subset \Gr(3,4,5)$ over $V_4$ defined by a pair $(\ker \Omega_1|_{V_4} + \ker \Omega_2|_{V_4}, V_4)$. 
\end{enumerate}
\end{proposition}
\begin{proof}
Case (1): From the proof of Proposition \ref{planes2}, we can obtain that $\rank\Omega_1$ and $\rank\Omega_2$ $\geq 2$, and the fiber of $T(Y^4)$ over $V_4$ is empty if $\rank\Omega_1|_{V_4}$ or $\rank\Omega_2|_{V_4}$ is $4$. Therefore, it enough to consider the case that $\rank\Omega_1|_{V_4}=\rank\Omega_2|_{V_4}=2$. 

Case (2): Assume that $\ker \Omega_1|_{V_4} = \ker \Omega_2|_{V_4}$. Since $\langle e_4 \rangle \subset \ker \Omega_1|_{V_4}$ and $\langle e_0 \rangle \subset \ker \Omega_2|_{V_4} $, we have $\ker \Omega_1|_{V_4} = \ker \Omega_2|_{V_4} = \langle e_0, e_4\rangle$. Then, for an element $ae_1 + be_2 + ce_3 \in V_4$, we have $c=b=0$ from the relation $\Omega_1|_{V_4} = \Omega_2|_{V_4} = 0$ which contradicts to the fact that $V_4$ is a $4$-dimensional vector space. Therefore $\ker \Omega_1|_{V_4}$ and $\ker \Omega_2|_{V_4}$ cannot be equal. Next, consider the case when $\ker \Omega_1|_{V_4} \cap \ker \Omega_2|_{V_4} = \langle v \rangle$, i.e. $1$-dimensional vector space generated by $v \in \CC^5$. If we write $v = a_0e_0 + \dots + a_4e_4$, then from the condition that $\Omega_1(v,e_0) = \Omega_2(v,e_4)=0$, we have $b_2=b_3=0$. Therefore we conclude that $\langle e_0, e_1, e_4 \rangle \subset V_4$. Conversely, if $\langle e_0, e_1, e_4 \rangle \subset V_4$, then we can observe that $\ker\Omega_1|_{V_4}\subset \langle e_0,e_1,e_4 \rangle$, $\ker\Omega_2|_{V_4}\subset \langle e_0,e_1,e_4 \rangle$ in the same manner. Therefore we have $\ker \Omega_1|_{V_4} \cap \ker \Omega_2|_{V_4}$ is a $1$-dimensional vector space. Hence, the locus where $\ker \Omega_1|_{V_4} \cap \ker \Omega_2|_{V_4}$ is $1$-dimensional is the image of the linear embedding $\mathrm{Gr}(1,2) \subset \mathrm{Gr}(4,5)$, given by the 1-1 correspondence between $1$-dimensional subspaces in $\CC^5/\langle e_0, e_1, e_4 \rangle $ and $4$-dimensional subspaces in $\CC^5$ containing $\langle e_0,e_1,e_4 \rangle$. Furthermore, when we consider a $4$-dimensional subspace $V_4 \in \mathrm{Gr}(1,2) \subset \mathrm{Gr}(4,5)$ of $\CC^5$, the fiber $T^{3,1}(Y^4) \subset F(1,4,5)$ over $V_4$ is represented by a pair $(\ker \Omega_1|_{V_4}\cap \ker\Omega_2|_{V_4}, V_4)$, and the fiber $T^{2,2}(Y^4) \subset F(3,4,5)$ over $V_4$ is represented by a pair $(\ker \Omega_1|_{V_4} + \ker \Omega_2|_{V_4}, V_4)$. It is obvious that the fiber of $T(Y^4)$ is empty over the $4$-dimensional subspace $V_4$ of $\CC^5$ where $\ker \Omega_1|_{V_4} \cap \ker \Omega_2|_{V_4} = \langle 0 \rangle$.
\end{proof}

\subsection{Clean intersection of $S(Y_4)$ and $T(\rG)$}\label{schint4}
This subsection is devoted to proving Theorem \ref{cleanint1} for the $4$-fold $Y_4$.

First we consider charts for $\bS(\rG)$. Since $\bS(\rG)$ is $\mathrm{Gr}(3,6)$-bundle over $\mathrm{Gr}(4,5)$, we should consider chart for $\Lambda \in \mathrm{Gr}(4,5)$ and $F \in \mathrm{Gr}(3,6) = \mathrm{Gr}(3,\wedge^2 \Lambda)$. There are $5$ standard charts for $\Lambda \in \mathrm{Gr}(4,5)$:
\[
\Lambda= 
\left(\begin{matrix}
1 & 0 & 0 & 0 & a \\
0 & 1 & 0 & 0 & b \\
0 & 0 & 1 & 0 & c \\
0 & 0 & 0 & 1 & d \\ 
\end{matrix}\right),
\Lambda= 
\left(\begin{matrix}
1 & 0 & 0 & a & 0 \\
0 & 1 & 0 & b & 0 \\
0 & 0 & 1 & c & 0 \\
0 & 0 & 0 & d & 1 \\ 
\end{matrix}\right),
\Lambda= 
\left(\begin{matrix}
1 & 0 & a & 0 & 0 \\
0 & 1 & b & 0 & 0 \\
0 & 0 & c & 1 & 0 \\
0 & 0 & d & 0 & 1 \\ 
\end{matrix}\right),
\] 
\[
\Lambda= 
\left(\begin{matrix}
1 & a & 0 & 0 & 0 \\
0 & b & 1 & 0 & 0 \\
0 & c & 0 & 1 & 0 \\
0 & d & 0 & 0 & 1 \\ 
\end{matrix}\right)
\textrm{ and }
\Lambda= 
\left(\begin{matrix}
a & 1 & 0 & 0 & 0 \\
b & 0 & 1 & 0 & 0 \\
c & 0 & 0 & 1 & 0 \\
d & 0 & 0 & 0 & 1 \\ 
\end{matrix}\right).
\]

By Proposition \ref{planes2} (and compare Remark \ref{planeeq}), we know that $T(Y)$ is the double cover over the linear embedding $\PP^1\cong \mathrm{Gr}(1,\CC^4/\langle e_0,e_1,e_4\rangle) \subset \mathrm{Gr}(4,5)$. Therefore, the intersection of $T(\rG)$ and $\bS(Y)$ only occurs in the following two charts:
\[
\Lambda=
\left(\begin{matrix}
1 & 0 & a & 0 & 0 \\
0 & 1 & b & 0 & 0 \\
0 & 0 & c & 1 & 0 \\
0 & 0 & d & 0 & 1 \\ 
\end{matrix}\right)
\textrm{ and }
\Lambda= 
\left(\begin{matrix}
1 & 0 & 0 & a & 0 \\
0 & 1 & 0 & b & 0 \\
0 & 0 & 1 & c & 0 \\
0 & 0 & 0 & d & 1 \\ 
\end{matrix}\right).
\]
and $a=b=d=0$ contained in the equations of $T(Y)$ in both cases, i.e. $a,b,d \in I_{T(Y)}$. Since $T(Y) = T^{2,2}(Y) \coprod T^{3,1}(Y)$, we can consider each part independently. We consider the clean intersection at $T(Y)_{3,1}$ firstly. Consider the first chart:
\[
\Lambda= 
\left(\begin{matrix}
1 & 0 & a & 0 & 0 \\
0 & 1 & b & 0 & 0 \\
0 & 0 & c & 1 & 0 \\
0 & 0 & d & 0 & 1 \\ 
\end{matrix}\right).
\]
Let $q_{01},...,q_{23}$ be a coordinate of a fiber of $\wedge^2\cU$ over this chart. Then we have $p_{12}-p_{03} = -aq_{01}-q_{02}+cq_{12}+dq_{13}$ and $p_{13}-p_{24} = -aq_{03}+q_{12}-bq_{13}-cq_{23}$. Each $\sigma_{3,1}$-plane in  $T(Y)_{3,1}$ which correspond to the vertex $x \in \PP\Lambda \subset \CC^5$ such that $(-q_{02}+cq_{12})(x,y)=0,(q_{12}-cq_{23})(x,y)=0$ (here, we consider $q_{ij}$ as a skew-symmetric two form) for all $y \in \Lambda$, because we have $a=b=d=0$ in $T^{3,1}(Y)$.
Then, by direct calculation, we can check that the sigma $\sigma_{3,1}$-plane correspond to the vertex $x$ contained in $T(Y)_{3,1}$ if and only if it satisfies the equations :
\[
\left( \begin{matrix}
x_0 & x_1 & x_2 & x_3
\end{matrix} \right)
\left( \begin{matrix}
0    & -ay_1 & -y_2  & 0    \\
ay_0 & 0     & cy_2  & dy_3 \\
y_0  & -cy_1 & 0     & 0    \\
0    & -dy_1 & 0     & 0    \\
\end{matrix} \right)
=0
\]
\[
\textrm{and }
\left( \begin{matrix}
x_0 & x_1 & x_2 & x_3
\end{matrix} \right)
\left( \begin{matrix}
0    & 0     & 0     & -ay_3 \\
0    & 0     & y_2   & -by_3 \\
0    & -y_1  & 0     & -cy_3 \\
ay_0 & by_1  & cy_2  & 0     \\
\end{matrix} \right)
=0.
\]
for all $y = (y_0,y_1,y_2,y_3) \in \Lambda$. Thus, we conclude that $x = [-c^2:-c:0:1] \in \PP\Lambda$.
Then, the corresponding $\sigma_{3,1}$-plane is spanned by $(-c^2,-c,0,1)\wedge(1,0,0,0),(-c^2,-c,0,1)\wedge(0,1,0,0),(-c^2,-c,0,1)\wedge(0,0,1,0)$. So we can rewrite it by a following $3\times 6$-matrix :
\[
\left(\begin{matrix}
-c  & 0   & 1 & 0 & 0 & 0 \\ 
c^2 & 0   & 0 & 0 & 1 & 0 \\ 
0   & c^2 & 0 & c & 0 & 1 \\
\end{matrix}\right).
\]
Thus, intersection of $\bS(Y)$ and $T(\rG)_{3,1}$ only occurs in the following chart of $F$:
\[
F=\left(\begin{matrix}
e & f & 1 & g & 0 & 0 \\ 
h & i & 0 & j & 1 & 0 \\ 
k & l & 0 & m & 0 & 1 \\
\end{matrix}\right).
\] 
In this chart, we can compute the ideal of $T^{3,1}(Y)$:
\[
T(Y)_{3,1} = \langle a,b,d,g,i,k,f,j,e+m,h-l,h-c^2,e+c \rangle
\]

On the other hand, $\sigma_{3,1}$-plane contained in this chart of $F$ is defined by the vertex of the form: $x = (\alpha,\beta,\gamma,1)$ which correspond to the following $3 \times 6$-matrix:
\[
\left(\begin{matrix}
\beta   & \gamma  & 1 & 0      & 0 & 0 \\ 
-\alpha & 0       & 0 & \gamma & 1 & 0 \\ 
0       & -\alpha & 0 & -\beta & 0 & 1 \\
\end{matrix}\right).
\]
Thus, we have $I_{T^{3,1}(\rG)} = \langle f-j, e+m, h-l, g,i,k \rangle$. On the other hand, from the equations 
$-aq_{01}-q_{02}+cq_{12}+dq_{13}$ and $-aq_{03}+q_{12}-bq_{13}-cq_{23}$, we obtain ideal for $\bS(Y)$:
\[
I_{\bS(Y)} = \langle -ae -f + cg, -ah - i + cj + d, -ak - l + cm,
-a + g, j-b,m-c  \rangle.
\]
Thus, we can check the clean intersection $I_{T(Y)_{3,1}} = I_{S(Y)} + I_{T(\rG)_{3,1}}$ in the first chart of $\Lambda$ by direct calculation.

Next, consider the second chart:
\[
\Lambda= 
\left(\begin{matrix}
1 & 0 & 0 & a & 0 \\
0 & 1 & 0 & b & 0 \\
0 & 0 & 1 & c & 0 \\
0 & 0 & 0 & d & 1 \\ 
\end{matrix}\right).
\]

Let $q_{01},...,q_{23}$ be a coordinate of a fiber of $\wedge^2\cU$ over this chart. Then we have $p_{12}-p_{03} = -bq_{01}-cq_{02}-dq_{03}+q_{12}$ and $p_{13}-p_{24} = -aq_{01}+cq_{12}+dq_{13}-q_{23}$. In the same manner as in the first chart case, we can show that $\sigma_{3,1}$-plane in $T(Y)^{3,1}$ correspond to the vertex $x = [1:c:0:-c^2] \in \PP\Lambda$.
The corresponding $\sigma_{3,1}$-plane is spanned by $(-c^2,-c,0,1)\wedge(0,1,0,0),(-c^2,-c,0,1)\wedge(0,0,1,0),(-c^2,-c,0,1)\wedge(0,0,0,1)$. So we can rewrite it by a following $3\times 6$-matrix:
\[
\left(\begin{matrix}
1 & 0 & 0 & 0 & c^2 & 0   \\ 
0 & 1 & 0 & c & 0   & c^2 \\ 
0 & 0 & 1 & 0 & c   & 0   \\
\end{matrix}\right).
\]
Thus, the intersection of $\bS(Y)$ and $T(\rG)_{3,1}$ only occurs in the following chart of $F\in Gr(3,\wedge^2 \Lambda)$:
\[
F=\left(\begin{matrix}
1 & 0 & 0 & e & f & g \\ 
0 & 1 & 0 & h & i & j \\ 
0 & 0 & 1 & k & l & m \\
\end{matrix}\right).
\] 
In this chart, we can compute the ideal of $T(Y)^{3,1}$:
\[
T(Y)_{3,1} = \langle a,b,d,g,i,k,e,m,h-l,f-j,h-c,f-c^2 \rangle.
\]

On the other hand, $\sigma_{3,1}$-plane contained in this chart of $F$ is defined by the vertex of the form $x = (1,\alpha,\beta,\gamma)$ which correspond to the following $3 \times 6$-matrix:
\[
\left(\begin{matrix}
1 & 0 & 0 & -\beta & -\gamma & 0       \\ 
0 & 1 & 0 & \alpha & 0       & -\gamma \\ 
0 & 0 & 1 & 0      & \alpha  & \beta   \\
\end{matrix}\right).
\]
Thus, we have $I_{T(\rG)_{3,1}} = \langle g,i,k,f-j,h-l,e+m \rangle$. On the other hand, from the equations 
$-bq_{01}-cq_{02}-dq_{03}+q_{12}$ and $-aq_{01}+cq_{12}+dq_{13}-q_{23}$, we obtain ideal for $\bS(Y)$:
\[
I_{S(Y)} = \langle -b+e,-c+h,-d+k,-a+ce+df-g,ch+di-j,ck+dl-m
 \rangle.
\]
So, we can check the clean intersection $I_{T(Y)_{3,1}} = I_{S(Y)} + I_{T(\rG)_{3,1}}$ in the second chart of $\Lambda$ by direct calculation. In summary, we checked clean intersection at $T(Y)_{2,2}$.

Next, we check clean intersection at $T^{2,2}(Y)$. Let us start with the first chart for $\Lambda$:
\[
\Lambda= 
\left(\begin{matrix}
1 & 0 & a & 0 & 0 \\
0 & 1 & b & 0 & 0 \\
0 & 0 & c & 1 & 0 \\
0 & 0 & d & 0 & 1 \\ 
\end{matrix}\right).
\]
Next $\sigma_{2,2}$-plane corresponds to $\PP^2$-plane in $\PP\Lambda\cong \PP^3 \subset \PP^4$ must one be of the following form (i.e. it correspond to the row space of the matrix $R\cdot \Lambda$): 
\[
R = 
\left( \begin{matrix}
 1 & 0 & 0 & \alpha \\ 
 0 & 1 & 0 & \beta \\
 0 & 0 & 1 & \gamma \\
\end{matrix}
\right)
or
\left( \begin{matrix}
 1 & 0 & \alpha  & 0 \\ 
 0 & 1 & \beta   & 0 \\
 0 & 0 & \gamma  & 1 \\
\end{matrix}
\right)
or
\left( \begin{matrix}
 1 & \alpha  & 0 & 0 \\ 
 0 & \beta   & 1 & 0 \\
 0 & \gamma  & 0 & 1 \\
\end{matrix}
\right)
or
\left( \begin{matrix}
 \alpha  & 1 & 0 & 0 \\ 
 \beta   & 0 & 1 & 0 \\
 \gamma  & 0 & 0 & 1 \\
\end{matrix}
\right).
\]
Therefore, the intersection of $\bS(Y)$ and $T(\rG)^{2,2}$ arises only in the following four charts of $F$: 
\[
F=
\bordermatrix{
&01& 02& 03& 12& 13& 23\cr
&1 & 0 & e & 0 & f & g \cr
&0 & 1 & h & 0 & i & j \cr
&0 & 0 & k & 1 & l & m 
}, 
F=
\bordermatrix{
&01& 02& 03& 12& 13& 23\cr
&1 & e & 0 & f & 0 & g \cr
&0 & h & 1 & i & 0 & j \cr
&0 & k & 0 & l & 1 & m 
},
\]
\[
F=
\bordermatrix{
&01& 02& 03& 12& 13& 23\cr
&e & 1 & 0 & f & g & 0 \cr
&h & 0 & 1 & i & j & 0 \cr
&k & 0 & 0 & l & m & 1 
}
\textrm{and }F=
\bordermatrix{
&01& 02& 03& 12& 13& 23\cr
&e & f & g & 1 & 0 & 0 \cr
&h & i & j & 0 & 1 & 0 \cr
&k & l & m & 0 & 0 & 1 
}
\]
where the upper indices are indices of Pl\"ucker coordinates.
Let us start with the first chart:
\[
F=\left(\begin{matrix}
1 & 0 & e & 0 & f & g \\
0 & 1 & f & 0 & i & j \\
0 & 0 & g & 1 & l & m \\ 
\end{matrix}\right). 
\]
In this case, we can easily observe that $\sigma_{2,2}$-plane contained in the intersection of $T(Y)_{2,2}$ and this chart must correspond to the row space of the matrix:
\[
R\Lambda = 
\left( \begin{matrix}
 1 & 0 & 0 & \alpha \\ 
 0 & 1 & 0 & \beta \\
 0 & 0 & 1 & \gamma \\
\end{matrix}
\right)\cdot \Lambda.
\]
We observe that this $\sigma_{2,2}$-plane which correspond to the row space of the matrix $R\Lambda$ correspond to the following matrix form in the chart of $F$:
\[
\left(\begin{matrix}
1 & 0 & \beta  & 0 & -\alpha & 0       \\
0 & 1 & \gamma & 0 & 0       & -\alpha \\
0 & 0 & 0      & 1 & \gamma  & -\beta  \\
\end{matrix}\right)
\]
But, in this case, the equations $-aq_{01}-q_{02}+cq_{12}+dq_{13}$ and $-aq_{03}+q_{12}-bq_{13}-cq_{23}$ does not have solutions since we have $a=b=d=0$ on $T^{2,2}(Y)$. Therefore, we can show that intersection of $S(Y)$ and $T^{2,2}(\rG)$ does not happens in the chart for $F$:
\[
F=
\left(\begin{matrix}
1 & 0 & e & 0 & f & g \\
0 & 1 & h & 0 & i & j \\
0 & 0 & k & l & l & m \\
\end{matrix} \right).
\]

In the similar manner, we can also show that no intersection of $S(Y)$ and $T^{2,2}(\rG)$ does not happens in the chart for $F$:
\[
F=
\left(\begin{matrix}
e & f & g & 1 & 0 & 0 \\
h & i & j & 0 & 1 & 0 \\
k & l & m & 0 & 0 & 1 \\
\end{matrix} \right).
\]

Therefore, it is enough to consider only two chart for $F$. Let us start with the following chart for $F$:
\[
F=
\left(\begin{matrix}
1 & e & 0 & f & 0 & g \\
0 & h & 1 & i & 0 & j \\
0 & k & 0 & l & 1 & m \\
\end{matrix} \right).
\]

In this case, we can easily observe that a $\sigma_{2,2}$-plane contained in this chart must correspond to the row space of a matrix of the form:
\[
R\Lambda = 
\left( \begin{matrix}
 1 & 0 & \alpha  & 0 \\ 
 0 & 1 & \beta   & 0 \\
 0 & 0 & \gamma  & 1 \\
\end{matrix}
\right)
\cdot \Lambda.
\]

Then, by the equation $-aq_{01}-q_{02}+cq_{12}+dq_{13}$ and $-aq_{03}+q_{12}-bq_{13}-cq_{23}$, we can observe that this $\sigma_{2,2}$-plane contained in $T(Y)_{2,2}$ if and only if it satisfies the following matrix equations:
\begin{align*}
& -a[R]^0\times[R]^1+c[R]^1\times[R]^2+d[R]^1\times[R]^3-[R]^0\times[R]^2=0 \textrm{ and} \\
& [R]^1\times[R]^2-a[R]^0\times[R]^3-b[R]^1\times[R]^3-c[R]^2\times[R]^3=0.
\end{align*}
Since we already have $a=b=d=0$ satisfied in $T^{2,2}(Y)$, by Proposition \ref{planes2}, the above equations reduce to
\begin{align*}
& c[R]^1\times[R]^2 - [R]^0\times[R]^2=0 \textrm{ and} \\
& [R]^1\times[R]^2 - c[R]^2\times[R]^3=0
\end{align*}
Therefore, we have
\begin{align*}
& c(-\gamma,0,\alpha) - (0,0,1) = 0 \\
& (-\gamma,0,\alpha) - c(1,0,0) = 0.
\end{align*}
Thus, there is no solution for these equations. So intersection of $T(\rG)_{2,2}$ and $S(Y)$ does not occur in this chart of $F$. So, in summary, we checked the clean intersection $I_{T(Y)_{2,2}} = I_{S(Y)} + I_{T(\rG)_{2,2}}$ in the first chart of $\Lambda$ and all chart of $F$. We can also check the clean intersection in the second chart for $\Lambda$:
\[
\Lambda= 
\left(\begin{matrix}
1 & 0 & 0 & a & 0 \\
0 & 1 & 0 & b & 0 \\
0 & 0 & 1 & c & 0 \\
0 & 0 & 0 & d & 1 \\ 
\end{matrix}\right).
\]
in the same manner, as we used in the case of the first chart of $\Lambda$. But since all process is parallel, we do not write it down here.

\bibliographystyle{alpha}

\begin{thebibliography}{CHL18}

\bibitem[BW89]{BW89}
J.G. Broida and S.G. Williamson.
\newblock {\em A Comprehensive Introduction to Linear Algebra}.
\newblock Advanced book program. Addison-Wesley, 1989.

\bibitem[CHL18]{CHL18}
Kiryong Chung, Jaehyun Hong, and SangHyeon Lee.
\newblock Geometry of moduli spaces of rational curves in linear sections of
  grassmannian $\text{Gr}(2,5)$.
\newblock {\em Journal of Pure and Applied Algebra}, 222(4):868 -- 888, 2018.

\bibitem[Chu22]{Chu22}
Kiryong Chung.
\newblock Desingularization of Kontsevich's compactification of twisted cubics
  in $V_5$.
\newblock {\em manuscripta mathematica}, 2022.

\bibitem[Chu23]{Chu23}
Kiryong Chung.
\newblock Double lines in the quintic del Dezzo fourfold.
\newblock {\em Bulletin of the Korean Mathematical Society}, 60(2):485--494, 03
  2023.

\bibitem[FN89]{FN89}
Mikio Furushima and Noboru Nakayama.
\newblock The family of lines on the fano threefold $V_5$.
\newblock {\em Nagoya Mathematical Journal}, 116:111--122, 1989.

\bibitem[FN72]{FN71}
Akira Fujiki and Shigeo Nakano.
\newblock Supplement to ``{O}n the inverse of monoidal transformation''.
\newblock {\em Publ. Res. Inst. Math. Sci.}, 7:637--644, 1971/72.

\bibitem[Fuj81]{Fuj81}
Takao Fujita.
\newblock {On the structure of polarized manifolds with total deficiency one,
  II}.
\newblock {\em Journal of the Mathematical Society of Japan}, 33(3):415 -- 434,
  1981.

\bibitem[Ili94]{Ili94}
Atanas Iliev.
\newblock The Fano surface of the Gushel threefold.
\newblock {\em Compositio Mathematica}, 94(1):81--107, 1994.

\bibitem[KPS18]{KPS18}
Alexander~G. Kuznetsov, Yuri~G. Prokhorov, and Constantin~A. Shramov.
\newblock Hilbert schemes of lines and conics and automorphism groups of fano
  threefolds.
\newblock {\em Japan. J. Math.}, pages 685--789, 2018.

\bibitem[Lee18]{Lee18}
Sanghyeon Lee.
\newblock {\em Geometry of moduli spaces of rational curves on Fano varieties}.
\newblock {Ph.D. Thesis}, Seoul National University, 2018.

\bibitem[Li09]{Li09}
Li~Li.
\newblock Wonderful compactification of an arrangement of subvarieties.
\newblock {\em Michigan Math. J.}, 58(2):535--563, 2009.

\bibitem[Pro94]{Pro94}
Yuri Prokhorov.
\newblock Compactifications of $\mathbb{C}^4$ of index $3$.
\newblock {\em In Algebraic geometry and its applications (Yaroslavl' , 1992)
  Aspects Math.}, E25:159--169, Vieweg, Braunschweig, 1994.

\bibitem[San14]{San14}
Giangiacomo Sanna.
\newblock Rational curves and instantons on the fano threefold $Y_5$.
\newblock {\em arXiv:1411.7994}, 2014.

\end{thebibliography}

\end{document}